\documentclass[11pt, a4paper,leqno]{amsart}
\usepackage{amsmath,amsthm,amscd,amssymb,amsfonts, amsbsy}
\usepackage{latexsym}
\usepackage{txfonts}
\usepackage{exscale}

\usepackage[colorlinks=true, pdfstartview=FitV, linkcolor=blue, citecolor=red, urlcolor=blue]{hyperref} 

\calclayout
\allowdisplaybreaks

\theoremstyle{plain}
\newtheorem{theorem}[equation]{Theorem}
\newtheorem{lemma}[equation]{Lemma}

\theoremstyle{definition}
\newtheorem{definition}[equation]{Definition}

\theoremstyle{remark}
\newtheorem{remark}[equation]{Remark}

\numberwithin{equation}{section}

\newcommand{\eps}{\varepsilon}

\newcommand{\dist}{\operatorname{dist}}

\newcommand{\dv}{\operatorname{div}}

\newcommand{\ree}{\mathbb{R}^{n+1}}

\newcommand{\Z}{\mathbb{Z}}
\newcommand{\dd}{\mathbb{D}}

\newcommand{\om}{\Omega}

\newcommand{\F}{\mathcal{F}}

\newcommand{\B}{\mathcal{B}}

\newcommand{\s}{\mathcal{S}}

\newcommand{\mut}{\mathfrak{m}}

\newcommand{\pom}{\partial\Omega}

\newcommand{\hm}{\omega}

\newcommand{\bbeta}{b\beta}

\renewcommand{\emptyset}{\mbox{\textup{\O}}}

\DeclareMathOperator{\diam}{diam}

\begin{document}
\allowdisplaybreaks

\title{A new characterization of chord-arc domains}


\address{Jonas Azzam \\ University of Washington \\
Department of Mathematics \\
Seattle, WA 98195-4350, USA}

\email{jazzam@math.washington.edu}


\address{Steve Hofmann
\\
Department of Mathematics
\\
University of Missouri
\\
Columbia, MO 65211, USA} \email{hofmanns@missouri.edu}


\address{Jos\'e Mar{\'\i}a Martell\\
Instituto de Ciencias Matem\'aticas CSIC-UAM-UC3M-UCM\\
Consejo Superior de Investigaciones Cient{\'\i}ficas\\
C/ Nicol\'as Cabrera, 13-15\\
E-28049 Madrid, Spain} \email{chema.martell@icmat.es}

\address{Kaj Nystr\"om\\Department of Mathematics, Uppsala University\\
S-751 06 Uppsala, Sweden}
\email{kaj.nystrom@math.uu.se}


\address{Tatiana Toro \\ University of Washington \\
Department of Mathematics \\
Seattle, WA 98195-4350, USA}

\email{toro@math.washington.edu}

\author{Jonas Azzam, Steve Hofmann,  Jos\'e Mar{\'\i}a Martell,\\ Kaj Nystr{\"o}m, Tatiana Toro}

\thanks{The first author was partially supported by NSF RTG grant 0838212.
The second author was supported by NSF grants DMS-1101244 and DMS-1361701.  The
third author was supported in part by MINECO Grant
MTM2010-16518, ICMAT Severo Ochoa project SEV-2011-0087. He
also acknowledges that the research leading to these results
as received funding from the European Research Council under
he European Union's Seventh Framework Programme
(FP7/2007-2013)/ ERC agreement no. 615112 HAPDEGMT. The
fourth author was partly supported by the Swedish research
council VR. The last author was partially supported by the Robert R. \& Elaine F. Phelps Professorship in Mathematics.
}

\date{\today}


\begin{abstract}
We show that if $\om \subset \ree$, $n\geq 1$,
is a uniform domain (aka 1-sided NTA domain), i.e., a domain which enjoys interior Corkscrew
and Harnack Chain conditions, then uniform rectifiability of the boundary of $\om$ implies the existence of
exterior Corkscrew points at all scales, so that in fact, $\om$ is a chord-arc domain,
i.e., a domain with an Ahlfors-David regular boundary which satisfies both interior and exterior Corkscrew conditions,
and an interior Harnack Chain condition.  We discuss some
implications of this result, for theorems of F. and M. Riesz type,
and for certain free boundary problems.
\end{abstract}

\maketitle



\section{Introduction and statement of main result}\label{s1}

An NTA ({\it Non-Tangentially Accessible}) domain $\om$ is one which enjoys an interior
Harnack Chain condition, as well as interior and exterior Corkscrew conditions
(see Definitions \ref{def1.cork}, \ref{def1.hc}, and \ref{def1.nta} below).  This notion was introduced
in \cite{JK}, and has found numerous applications in the theory of elliptic equations and free boundary problems.
A {\it chord-arc} domain is an NTA domain whose boundary is Ahlfors-David regular (ADR)
(see Definition \ref{def1.ADR} below).

A 1-sided NTA domain (also known as a {\it uniform domain} in the literature) satisfies
interior Corkscrew and Harnack Chain conditions, but one imposes no assumptions
on the exterior domain $\om_{\rm ext}:=\ree\setminus \overline{\om}$.   A 1-sided chord-arc domain
is a 1-sided NTA domain with an ADR boundary.  In general, the 1-sided conditions are strictly weaker than their
standard counterparts:  for example, the domain $\om =\mathbb{R}^2\setminus K$, where $K$ is
Garnett's ``4-corners Cantor set" (see, e.g., \cite{DS2}), is a 1-sided chord-arc domain, but
$\overline{\om}=\mathbb{R}^2$, thus, there is no exterior domain.

The various versions of non-tangential accessibility have been studied in connection
with quantitative analogues of the F. and M. Riesz Theorem \cite{RR}, in which one
obtains scale-invariant absolute continuity (e.g., the $A_\infty$ condition of Muckenhoupt,
or more generally, weak-$A_\infty$)
of harmonic measure $\hm$ with respect to surface measure on the boundary, given some quantitative
rectifiability property of the boundary, along with some quantitative connectivity hypothesis of
$\om$.  This has been done in 2 dimensions by M. Lavrentiev \cite{La}; and in higher dimensions, for Lipschitz domains
by B. Dahlberg  \cite{D}; for chord-arc domains by G. David and D. Jerison \cite{DJe}, and
independently, by S. Semmes \cite{Se}; for NTA domains without an ADR hypothesis by
M. Badger \cite{Ba}; for the Riesz
measure associated to the $p$-Laplacian in chord-arc domains, by J. Lewis and the fourth
named author of this paper \cite{LN}; for domains with ADR boundaries, satisfying
an ``interior big pieces of Lipschitz
sub-domains" condition,
by B. Bennewitz and J. Lewis \cite{BL};
and for 1-sided chord-arc domains with uniformly rectifiable boundaries
(see Definition \ref{def1.UR} below), by the second and third named authors of this paper \cite{HM-URHM}.   In addition,
the second and third authors of this paper, jointly with I. Uriarte-Tuero \cite{HMU}, have obtained
a free boundary result that is a converse to the theorem of \cite{HM-URHM}, in which scale invariant
absolute continuity of harmonic measure with respect to surface measure, in the presence of
the 1-sided chord-arc condition, implies uniform rectifiability of $\pom$.

It turns out, perhaps surprisingly, in light of the counter-example of T. Hrycak
(see \cite{DS2}), that the result of \cite{HM-URHM} may be subsumed in that of \cite{DJe} and \cite{Se},
while on the other hand, the free boundary result of \cite{HMU} is stronger than the authors had initially realized.
Our main result in this paper is the following.  
\begin{theorem}\label{t1}
Suppose that $\om\subset \ree$ is a
uniform (aka 1-sided NTA) domain,
and that $\pom$ is uniformly rectifiable.  Then $\om$ is a chord-arc domain.
\end{theorem}

Let us point out that both of our hypotheses are essential, in the sense that neither one,
alone, allows one to draw the stated conclusion.  Indeed,
as we have noted above, a 1-sided NTA domain, with
ADR boundary, need not be NTA in general:
in fact there may be no exterior domain.
On the other hand, Hrycak's example shows that
uniform rectifiability, in general, is strictly weaker than a
``Big Pieces of Lipschitz Graphs" condition\footnote{However,
as shown by the first named author of this paper and R. Schul \cite{AS},
uniform rectifiability is equivalent to a ``Big Pieces of Big Pieces of Lipschitz Graphs'' condition.},
whereas the latter property is enjoyed
by boundaries of chord-arc domains, by the main geometric result of \cite{DJe}.

Theorem \ref{t1}, combined with previous work mentioned above,
has implications for theorems of F. and M. Riesz type, and for
certain free boundary problems.  We now have the following.
\begin{theorem}\label{t2}
Suppose that $\om\subset \ree$ is a uniform (aka 1-sided NTA) domain, whose boundary is Ahlfors-David regular.
Then the following are equivalent:
\begin{enumerate}\itemsep=0.05cm
\item $\pom$ is uniformly rectifiable.
\item $\om$ is an NTA domain, and hence, a chord-arc domain.
\item $\hm\in A_\infty$.
\item $\hm \in $ weak-$A_\infty$.
\end{enumerate}
\end{theorem}
Here $\hm$ denotes harmonic measure for $\om$ with some fixed pole, and the statements (3) and (4)
are understood in a scale invariant sense (see, e.g., \cite{HM-URHM, HMU}).  That (2) implies (3)
was proved in \cite{DJe}, and independently, in \cite{Se}; it is of course trivial that (3) implies (4),
while (4) implies (1) is the main result of \cite{HMU}.  Theorem \ref{t1} closes the circle by establishing
that (1) implies (2).  In particular, this yields a sharpened version of the free boundary result of \cite{HMU}, namely,
we now
obtain that (4) implies (2), under the stated background hypotheses.

In connection with the previous result we recall the David-Semmes conjecture recently proved in \cite{NToV} (see also \cite{HMM} for the case of a boundary of a 1-sided NTA domain with ADR boundary) which establishes that uniform rectifiability is equivalent to the boundedness of the Riesz transform. Hence Theorem \ref{t1} implies that, under the same background hypothesis, the Riesz transform is bounded on $L^2(\pom)$ if and only if $\Omega$ is an NTA domain and consequently a chord-arc domain.

In Section 2, we establish notation and review some necessary definitions and
preliminary results. In particular, we recall two different
characterizations of uniform rectifiability, one {\it analytic} and the other
{\it geometric}. In Sections 3 and 4, we give two separate proofs of
Theorem \ref{t1}, each one using a different characterization.
Finally, in an appendix, we provide the proof of a folklore theorem concerning the equivalence
of 1-sided NTA  and uniform domains.

\section{Preliminaries}\label{section:prelim}

\subsection{Notation and conventions}

\begin{list}{$\bullet$}{\leftmargin=0.4cm  \itemsep=0.2cm}

\item We use the letters $c,C$ to denote harmless positive constants, not necessarily the same at each occurrence, which depend only on dimension and the
constants appearing in the hypotheses of the theorems (which we refer to as the ``allowable parameters'').  We shall also sometimes write $a\lesssim b$ and $a \approx b$ to mean, respectively, that $a \leq C b$ and $0< c \leq a/b\leq C$, where the constants $c$ and $C$ are as above, unless
explicitly noted to the contrary.   Unless otherwise specified upper case constants are greater than $1$  and lower case constants are smaller than $1$.

\item Given a domain $\Omega \subset \ree$, we shall
use lower case letters $x,y,z$, etc., to denote points on $\partial \Omega$, and capital letters
$X,Y,Z$, etc., to denote generic points in $\ree$ (especially those in $\ree\setminus \partial\Omega$).

\item The open $(n+1)$-dimensional Euclidean ball of radius $r$ will be denoted
$B(x,r)$ when the center $x$ lies on $\partial \Omega$, or $B(X,r)$ when the center
$X \in \ree\setminus \partial\Omega$.  A {\it surface ball} is denoted
$\Delta(x,r):= B(x,r) \cap\partial\Omega.$

\item If $\pom$ is bounded, it is always understood (unless otherwise specified) that all surface balls have radii controlled by the diameter of $\pom$: that is if $\Delta=\Delta(x,r)$ then $r\lesssim \diam(\pom)$. Note that in this way $\Delta=\pom$ if $\diam(\pom)<r\lesssim \diam(\pom)$.



\item For $X \in \ree$, we set $\delta(X):= \dist(X,\partial\Omega)$.

\item We let $H^n$ denote $n$-dimensional Hausdorff measure, and let
$\sigma := H^n\big|_{\partial\Omega}$ denote the surface measure on $\partial \Omega$.

\item For a Borel set $A\subset \ree$, we let $1_A$ denote the usual
indicator function of $A$, i.e. $1_A(x) = 1$ if $x\in A$, and $1_A(x)= 0$ if $x\notin A$.



\item For a Borel subset $A\subset\partial\Omega$, we
set $\fint_A f d\sigma := \sigma(A)^{-1} \int_A f d\sigma$.

\item We shall use the letter $I$ (and sometimes $J$)
to denote a closed $(n+1)$-dimensional Euclidean cube with sides
parallel to the co-ordinate axes, and we let $\ell(I)$ denote the side length of $I$.
We use $Q$ to denote a dyadic ``cube''
on $\partial \Omega$.  The
latter exist, given that $\partial \Omega$ is ADR  (cf. \cite{DS1}, \cite{Ch}), and enjoy certain properties
which we enumerate in Lemma \ref{lemma:Christ} below.
\end{list}

\subsection{Some definitions}\label{ssdefs} 

\begin{definition}[\bf Ahlfors-David regular]\label{def1.ADR}
 We say that a closed set $E \subset \ree$ is $n$-dimensional ADR (or simply ADR) (Ahlfors-David regular) if
there is some uniform constant $C$ such that
\begin{equation} \label{eq1.ADR}
\frac1C\, r^n \leq H^n(E\cap B(x,r)) \leq C\, r^n,\,\,\,\forall r\in(0,R_0),x \in E,
\end{equation}
where $R_0$ is the diameter
of $E$ (which may be infinite).  
\end{definition}

\begin{definition}[\bf Uniform Rectifiability]\label{def1.UR}
Following David
and Semmes \cite{DS1}, \cite{DS2}, we say that a closed set $E\subset \ree$ is $n$-dimensional UR (or simply UR) (Uniformly Rectifiable), if
it satisfies the ADR condition \eqref{eq1.ADR}, and if for some uniform constant $C$ and for every
Euclidean ball $B:=B(x_0,r), \, r\leq \diam(E),$ centered at any point $x_0 \in E$, we have the Carleson measure estimate
\begin{equation}\label{eq1.sf}\iint_{B}
|\nabla^2 \mathcal{S}1(X)|^2 \,\dist(X,E) \,dX \leq C r^n,
\end{equation}
where $\mathcal{S}f$ is the harmonic single layer potential of $f$, i.e.,
\begin{equation}\label{eq1.layer}
\mathcal{S}f(X) :=c_n\, \int_{E} |X-y|^{1-n} f(y) \,dH^n(y).
\end{equation}
Here, the normalizing constant $c_n$ is chosen so that
$\mathcal{E}(X) := c_n |X|^{1-n}$ is the usual fundamental solution for the Laplacian
in $\ree$ (one should use a logarithmic potential in ambient dimension $n+1=2$).
\end{definition}

A geometric characterization of uniform rectifiability can be given in terms of the so called
{\it bilateral $\beta$-numbers}. For $x\in E$, a hyperplane $P$,
and $r>0$ we set \[\bbeta_{E}(x,r,P)=r^{-1}\left(\sup_{y\in
E\cap B(x,r)}\dist(y,P)+ \sup_{y\in P\cap
B(x,r)}\dist(y,E)\right)\]
and then define
\[\bbeta_{E}(x,r)=\inf_{P}\bbeta_E(x,r,P)\]
where the infimum is over all $n$-dimensional hyperplanes
$P\subseteq\mathbb{R}^{n+1}$.

\begin{definition}[\bf BWGL]\label{defi:BWGL}
We say that an $n$-dimensional ADR set $E$
satisfies the {\it bilateral weak geometric lemma} or {\it
BWGL} if, for each $\varepsilon>0$,
the set
\[
\widehat{B}_{\varepsilon}:=
\big\{(x,r):x\in E, r>0,\bbeta_E(x,r)\geq\varepsilon\big\}
\]
is a {\it Carleson set}, i.e.,  there is $C_{1}>0$ so that if we define
\[
\widehat{\sigma}(A)=\iint_{A}dH^n\frac{dt}{t},
\qquad A\subset E\times(0,\infty),
\]
then
\begin{equation}\label{e:carleson}
\widehat{\sigma}
\left(\widehat{B}_{\varepsilon}\cap \big(B(x,r)\times (0,r)\big)\right)
\leq C_{1}r^{n}
\end{equation} for all $x\in E$ and $r>0$.
\end{definition}

\begin{theorem}[{\cite[Theorem 2.4, Part I]{DS2}}]\label{theor:UR-2}
An $n$-dimensional ADR set
$E$ is uniformly rectifiable if and only if it
satisfies the BWGL.
\end{theorem}

\begin{remark} We note that there are numerous characterizations of uniform rectifiability
given in \cite{DS1,DS2};  the two stated above will be most useful for our purposes, and
appear in \cite[Chapter 2, Part I]{DS2} and \cite[Chapter 3,  Part III]{DS2}.
\end{remark}

We recall
that the UR sets
are precisely those for which all ``sufficiently nice'' singular integrals
are bounded on $L^2$ (see \cite{DS1}).  We further remark that uniform rectifiability
is the scale invariant version of
rectifiability:
in particular, BWGL may be viewed
as a quantitative version of the characterization of rectifiable sets in terms of the
existence a.e. of approximate tangent planes.  In this context, see also \cite{J1}.

\begin{definition}[\bf Corkscrew condition]\label{def1.cork}
Following \cite{JK}, we say that a domain $\Omega\subset \ree$
satisfies the {\it Corkscrew condition} if for some uniform constant $c>0$ and
for every surface ball $\Delta:=\Delta(x,r),$ with $x\in \partial\Omega$ and
$0<r<\diam(\partial\Omega)$, there is a ball
$B(X_\Delta,cr)\subset B(x,r)\cap\Omega$.  The point $X_\Delta\subset \Omega$ is called
a {\it corkscrew point relative to} $\Delta,$ (or, relative to $B$). We note that  we may allow
$r<C\diam(\pom)$ for any fixed $C$, simply by adjusting the constant $c$.
\end{definition}

\begin{definition}[\bf Harnack Chain condition]\label{def1.hc}
Again following \cite{JK}, we say
that $\Omega$ satisfies the {\it Harnack Chain condition} if there is a uniform constant $C$ such that
for every $\rho >0,\, \Lambda\geq 1$, and every pair of points
$X,X' \in \Omega$ with $\delta(X),\,\delta(X') \geq\rho$ and $|X-X'|<\Lambda\,\rho$, there is a chain of
open balls
$B_1,\dots,B_N \subset \Omega$, $N\leq C(\Lambda)$,
with $X\in B_1,\, X'\in B_N,$ $B_k\cap B_{k+1}\neq \emptyset$
and $C^{-1}\diam (B_k) \leq \dist (B_k,\partial\Omega)\leq C\diam (B_k).$  The chain of balls is called
a {\it Harnack Chain}.
\end{definition}

\begin{definition}[\bf 1-sided NTA]\label{def1.1nta}
If $\Omega$ satisfies both the Corkscrew and Harnack Chain conditions, then we say that
$\Omega$ is a {\it 1-sided NTA domain}.
\end{definition}

\begin{remark}
We observe that the 1-sided NTA condition is a quantitative connectivity condition.
\end{remark}

An alternative (and quantitatively equivalent)
definition is as follows.

\begin{definition}[\bf Uniform domain]\label{unif-dom}
For $0<c<1<C$, we say that an open set $\Omega\subseteq\mathbb{R}^{n+1}$
is a {\it $(c,C)$-uniform domain} if, for any two
points $X,Y\in \Omega$, there is a path $\gamma$ so that
\begin{enumerate}\itemsep=0.05cm
  \item $\ell(\gamma)\leq C|X-Y|$, where $\ell(\gamma)$ denotes the length of
  $\gamma$, and
  \item for any $Z\in\gamma$, $\delta(Z)\geq
  c\dist(Z,\{X,Y\})$.
\end{enumerate}
We call such a curve a {\it good curve} for $X$ and $Y$.
\end{definition}

It is well known that Definitions \ref{def1.1nta} and \ref{unif-dom}
are equivalent, nevertheless finding a precise reference is difficult. That every uniform domain is a 1-sided NTA domain was proved by Gehring and Osgood
\cite{GO}.   Indeed,  the {\it extension domains} of P. Jones
\cite{J2} are clearly 1-sided NTA by definition, while in
\cite{GO}, the authors prove that extension domains are the same as uniform domains
(see also \cite{V}).
For a more direct proof of the fact that uniform domains are 1-sided NTA domains,
see \cite[Lemma 4.2 and 4.3]{BS},  as well as their proofs. For completeness,
we include the converse implication below.

\begin{theorem}\label{1-NTA-to-uniform}
If $\Omega$ is a 1-sided NTA domain then $\Omega$ is a uniform domain
with constants that only depend on the 1-sided NTA constants.
\end{theorem}

We defer the proof of Theorem \ref{1-NTA-to-uniform} to an appendix.

\begin{definition}[\bf NTA domain]\label{def1.nta}
Following \cite{JK}, we say that a domain
$\Omega$ is an {\it NTA  domain} if
it is a 1-sided NTA domain and if, in addition, $\om_{\rm ext}:= \ree\setminus \overline{\Omega}$
also satisfies the Corkscrew condition.
\end{definition}

\begin{definition}[\bf Chord-arc domain]
$\Omega$ is a {\it chord-arc domain} if
it is an NTA domain with an ADR boundary.
\end{definition}



\subsection{Dyadic grids}

\begin{lemma}[\bf Existence and properties of the ``dyadic grid'']\label{lemma:Christ}
\cite{DS1,DS2}, \cite{Ch}.
Suppose that $E\subset \ree$ satisfies
the ADR condition \eqref{eq1.ADR}.  Then there exist
constants $ a_0>0,\, \eta>0$ and $C_1<\infty$, depending only on dimension and the
ADR constants, such that for each $k \in \mathbb{Z},$
there is a collection of Borel sets (``cubes'')
$$
\mathbb{D}_k:=\{Q_{j}^k\subset E: j\in \mathfrak{I}_k\},$$ where
$\mathfrak{I}_k$ denotes some (possibly finite) index set depending on $k$, satisfying

\begin{list}{$(\theenumi)$}{\usecounter{enumi}\leftmargin=.8cm
\labelwidth=.8cm\itemsep=0.2cm\topsep=.1cm
\renewcommand{\theenumi}{\roman{enumi}}}

\item $E=\cup_{j}Q_{j}^k\,\,$ for each
$k\in{\mathbb Z}$.

\item If $m\geq k$ then either $Q_{i}^{m}\subset Q_{j}^{k}$ or
$Q_{i}^{m}\cap Q_{j}^{k}=\emptyset$.

\item For each $(j,k)$ and each $m<k$, there is a unique
$i$ such that $Q_{j}^k\subset Q_{i}^m$.

\item Diameter $\left(Q_{j}^k\right)\leq C_12^{-k}$.

\item Each $Q_{j}^k$ contains some surface ball $\Delta \big(x^k_{j},a_02^{-k}\big):=
B\big(x^k_{j},a_02^{-k}\big)\cap E$.

\item $H^n\left(\left\{x\in Q^k_j:{\rm dist}(x,E\setminus Q^k_j)\leq \tau \,2^{-k}\right\}\right)\leq
C_1\,\tau^\eta\,H^n\left(Q^k_j\right),$ for all $k,j$ and for all $\tau\in (0,a_0)$.
\end{list}
\end{lemma}

A few remarks are in order concerning this lemma.

\begin{list}{$\bullet$}{\leftmargin=0.4cm  \itemsep=0.2cm}

\item In the setting of a general space of homogeneous type, this lemma has been proved by Christ
\cite{Ch}, with the
dyadic parameter $1/2$ replaced by some constant $\delta \in (0,1)$.
In fact, one may always take $\delta = 1/2$ (cf.  \cite[Proof of Proposition 2.12]{HMMM}).
In the presence of the Ahlfors-David
property (\ref{eq1.ADR}), the result already appears in \cite{DS1,DS2}.

\item  For our purposes, we may ignore those
$k\in \mathbb{Z}$ such that $2^{-k} \gtrsim {\rm diam}(E)$, in the case that the latter is finite.

\item  We shall denote by  $\mathbb{D}=\mathbb{D}(E)$ the collection of all relevant
$Q^k_j$, i.e., $$\mathbb{D} := \cup_{k} \mathbb{D}_k,$$
where, if $\diam (E)$ is finite, the union runs
over those $k$ such that $2^{-k} \lesssim  {\rm diam}(E)$.

\item Given a cube $Q\in \dd$, we set
 \begin{equation}\label{eq2.discretecarl}
\dd_Q:= \left\{Q'\in \dd: Q'\subseteq Q\right\},
\end{equation}

\item For a dyadic cube $Q\in \mathbb{D}_k$, we shall
set $\ell(Q) = 2^{-k}$, and we shall refer to this quantity as the ``length''
of $Q$.  Evidently, $\ell(Q)\approx \diam(Q).$


\item Properties $(iv)$ and $(v)$ imply that for each cube $Q\in\mathbb{D}_k$,
there is a point $x_Q\in E$, a Euclidean ball $B(x_Q,r_Q)$ and a surface ball
$\Delta(x_Q,r_Q):= B(x_Q,r_Q)\cap E$ such that
$c\ell(Q)\leq r_Q\leq \ell(Q)$, for some uniform constant $c>0$,
and \begin{equation}\label{cube-ball}
\Delta(x_Q,2r_Q)\subset Q \subset \Delta(x_Q,Cr_Q),\end{equation}
for some uniform constant $C$.
We shall denote this ball and surface ball by
\begin{equation}\label{cube-ball2}
B_Q:= B(x_Q,r_Q) \,,\qquad\Delta_Q:= \Delta(x_Q,r_Q),\end{equation}
and we shall refer to the point $x_Q$ as the ``center'' of $Q$.


\end{list}

It will be useful  to dyadicize the Corkscrew condition, and to specify
precise Corkscrew constants.
Let us now specialize to the case that  $E=\pom$ is ADR,
with $\Omega$ satisfying the Corkscrew condition.
Given $Q\in \mathbb{D}(\partial\Omega)$, we
shall sometimes refer to a ``Corkscrew point relative to $Q$'', which we denote by
$X_Q$, and which we define to be the corkscrew point $X_\Delta$ relative to the surface ball
$\Delta:=\Delta_Q$ (see \eqref{cube-ball}, \eqref{cube-ball2} and Definition \ref{def1.cork}).  We note that
\begin{equation}\label{eq1.cork}
\delta(X_Q) \approx \dist(X_Q,Q) \approx \diam(Q).
\end{equation}

\medskip

\begin{definition} ({\bf $c_0$-exterior
Corkscrew condition}).  \label{def1.dyadcork} Fix  a constant $c_0\in (0,1)$,
and a domain $\Omega\subset \ree$, with ADR boundary.  We say
that a cube $Q\in \dd(\pom)$ satisfies the
the $c_0$-exterior Corkscrew condition, if
there is a point $z_Q\in \Delta_Q$, and a point $X^-_Q\in B(z_Q,r_Q/4) \setminus \overline{\Omega}$, such that
$B(X^-_Q,\,c_0\,\ell(Q))\subset B(z_Q,r_Q/4)\setminus \overline{\Omega}$, where
$\Delta_Q=\Delta(x_Q,r_Q)$ is the surface ball defined above in \eqref{cube-ball}--\eqref{cube-ball2}.
\end{definition}


\section{The analytic proof of Theorem \ref{t1}}\label{s4}

We suppose that $\om$ is a 1-sided NTA (i.e., uniform) domain,
with uniformly rectifiable boundary.  It suffices to show that $\om$ satisfies an exterior Corkscrew condition
at all scales (up to the diameter of $\pom$).

To this end, we define a discrete measure $\mut$ as follows.
Let $\B=\B(c_0)$
denote the collection
of $Q\in\dd(\pom)$ for which the $c_0$-exterior Corkscrew condition (see
Definition \ref{def1.dyadcork}) fails.
Set
\begin{equation}\label{eq2.5}
\alpha_Q := \left\{
\begin{array}{ll}
\sigma(Q)\,,&
\,\,{\rm if\,} Q \in \B,
\\[6pt]
0\,,&\,\,
{\rm otherwise}\,.
\end{array}
\right.
\end{equation}
For any subcollection $\dd'\subset\dd(\pom)$, we set
\begin{equation}\label{mut-def}
\mut(\dd'):= \sum_{Q\in\dd'}\alpha_{Q}.
\end{equation}

We will prove, as a consequence of the UR  and 1-sided NTA properties, that the collection $\B$ satisfies a packing condition, i.e., that $\mut$ is a discrete
Carleson measure, provided that $c_0$ is small enough.
More precisely, we have the following.
\begin{lemma}\label{pack}  Let $\Omega$ be 
a 1-sided NTA domain with UR boundary, and let $\B=\B(c_0)\subset \dd$ be the
collection defined above.
Then there is a $c_0$ sufficiently small, such that
the measure $\mut$ satisfies the packing condition
\begin{equation}\label{eq4.4}
\sup_{Q\in\dd} \,\frac{ \mut(\dd_Q)}{\sigma(Q)}\,  =\, \sup_{Q\in\dd}\,\frac1{\sigma(Q)}
\sum_{Q'\in\B:\, Q'\subset Q} \sigma(Q') \,\leq M_1\,,
\end{equation}
where the  constants $c_0$ and $M_1$  depend only upon
dimension, and on the ADR/UR and 1-sided NTA constants.
\end{lemma}

Let us momentarily take the lemma for granted, and deduce the conclusion of
Theorem \ref{t1}.    We fix a cube $Q\in\dd(\pom)$, and we seek to show that
$\om_{\rm ext}$ has a Corkscrew point relative to $Q$.
Let $\Delta_Q\subset Q$ denote the surface ball defined in \eqref{cube-ball}--\eqref{cube-ball2},
and let $Q_1$ be a sub-cube of $Q$, of maximal size, that is contained in $\Delta_Q$.
We then have that $\ell(Q_1)\geq c\ell(Q)$.
By \eqref{eq4.4} (applied to $Q_1$), and the ADR condition,
there is a constant $c_1$, depending only on $M_1$ and the ADR constants, and
a cube $Q'\in \dd_{Q_1}\setminus \B$, with $\ell(Q')\geq c_1 \ell(Q)$.
Since $Q'\notin \B$, it therefore enjoys the $c_0$-exterior Corkscrew condition,
and therefore, so does $Q$, but with $c_0$ replaced by $c_0'=c_0c_1$.
Since every surface ball contain a cube of comparable diameter, this means that there is an exterior Corkscrew point relative to every surface ball on the boundary, and therefore $\om$ is NTA, and hence chord-arc.

It remains to prove the lemma.
\begin{proof}[Proof of Lemma \ref{pack}]
We follow a related argument in \cite{HM-URHM}, which in turn uses an
idea from \cite{DS2}.
Let $\B$ denote the collection
of $Q\in\dd$ for which the $c_0$-exterior Corkscrew condition (cf.
Definition \ref{def1.dyadcork})
fails.
We fix a cube $Q\in \B$,  a point $z_Q\in \Delta_Q\subset Q$, and set $B=B(z_Q,r/4)$,
with $r=r_Q\approx \ell(Q)$,
and set $\Delta=B\cap \pom$.
Let $\Phi\in C^\infty_0(B)$,
with $0\leq\Phi\leq 1$, $\Phi\equiv 1$ on $(1/2)B$,
and $\|\nabla\Phi\|_\infty \lesssim r^{-1}$.  Let $\mathcal{L}:=-\dv\cdot\nabla$
denote the usual Laplacian in $\ree$, and let $\s$ denote the single layer potential
for $\mathcal{L}$, as in \eqref{eq1.layer} with $E=\pom$.
Since $\pom$ is ADR,
\begin{multline}\label{eq2.21}
\sigma(\Delta) \approx \int_{\pom}\Phi \,d\sigma
=\left\langle \mathcal{L} \s1,\Phi\right\rangle
= \iint_{\ree}(\nabla \s1(X)-\vec{\beta})\cdot \nabla\Phi(X)\, dX\\[4pt]
\lesssim \frac1{r}\left(
 \iint_{B\setminus \om}|\nabla \s1(X)-\vec{\beta}\,| dX \,+\,
  \iint_{B\cap\om}|\nabla \s1(X)-\vec{\beta}\,| dX\right)\\[4pt]
  =:\,\frac1{r}\,\big(I+II\big)\,,
\end{multline}
where $\vec{\beta}$ is a constant vector at our disposal.  We let $B^*:= \kappa_0 B$
be a concentric dilate of $B$ with $\kappa_0$ a large constant (see \cite[(5.12)]{HM-URHM}), and set $\Delta^*:= B^*\cap\pom$.
As in the proof of \cite[Lemma 5.10]{HM-URHM}, we may choose $\vec{\beta}$
so that the following properties hold:
\begin{equation}\label{eq2.22}
\|\nabla \s 1_{\pom\setminus \Delta^*} -\vec{\beta}\,\|_{L^\infty(B) }\leq C\,,
\end{equation}
and
\begin{multline}\label{eq2.23}
r^{-1}   \iint_{B\cap(\om\setminus \Sigma_{B,\eps})}|\nabla \s1(X)-\vec{\beta}\,| dX\\[4pt]
\leq\, C_\eps \,r^{n/2}\left(\iint_{U_{B,\eps}}
|\nabla^2\s1(X)|^2 \delta(X) dX\right)^{1/2}\,,
\end{multline}
where $\Sigma_{B,\eps}$ is a ``border strip" of thickness $C\eps r$, and
$U_{B,\eps}$ is a Whitney region with $\eps r\lesssim \delta(X)\approx \dist(X,\Delta)
 \lesssim r$, for all $X\in U_{B,\eps}$, with $\eps$ a small positive constant
 to be chosen.  In the language of \cite[Sections 4 and 5]{HM-URHM},  $\Sigma_{B,\eps}=\Omega\setminus \Omega_{\F(\epsilon\,r),Q}$ with $\F=\emptyset$,
$U_{B,\eps}=\Omega_{\F(\epsilon\,r),Q}^{\rm fat}$ and $\vec{\beta}$ is the average of $\nabla S1$ on $\Omega_{\F(\epsilon\,r),Q}$.
 We observe that
 \eqref{eq2.23} is a consequence of a Poincar\'e inequality proved in \cite[Section 4]{HM-URHM}.
 Moreover, by \cite[Lemma 5.1, Lemma 5.3, and Corollary 5.6]{HM-URHM}, we have
  \begin{equation}\label{eq2.24}
  \iint_{B}|\nabla \s1_{\Delta^*}(X)|^q dX \leq C_q \, r^{n+1}\,,\quad 1\leq q<(n+1)/n\,,
  \end{equation}
  \begin{equation}\label{eq2.25}
  |\Sigma_{B,\eps} \cap B|\lesssim \eps r^{n+1}\,,
  \end{equation}
  and
 \begin{equation}\label{eq2.26}
  \iint_{B\cap \Sigma_{B,\eps}}|\nabla \s1_{\Delta^*}(X)| dX \lesssim \eps^\gamma r^{n+1}\,,
  \end{equation}
for some fixed $\gamma$, with $0<\gamma<1$.
Combining these facts, we see that
$$\frac{II}{r} \lesssim \eps^\gamma \sigma(\Delta) \,+\,  C_\eps\,  \sigma(\Delta)^{1/2}\left(\iint_{U_{B,\eps}}
|\nabla^2\s1(X)|^2 \delta(X) dX\right)^{1/2}\,.$$
Furthermore, by \cite[Lemma 5.7]{HM-URHM}, the failure of the $c_0$-exterior Corkscrew property
implies that
$$|B\setminus \om|\lesssim c_0 r^{n+1}\,.$$
Combining the latter fact with \eqref{eq2.22} and  \eqref{eq2.24}, we find that
$$\frac{I}{r} \,\lesssim \, c_0^{1/q'}\sigma(\Delta)\,.$$
If $\eps$ and $c_0$ are chosen small enough, then the small terms
may be hidden on the left hand side of \eqref{eq2.21},
to obtain that
\begin{equation}\label{eq2.27}
\sigma(\Delta) \lesssim \iint_{U_{B,\eps}}
|\nabla^2\s1(X)|^2 \delta(X) dX\,.
\end{equation}
As observed above, the measure $dm(X):=|\nabla^2\s1(X)|^2 \delta(X) dX$
is a Carleson measure in $\ree\setminus \pom$, since $\pom$ is UR (see \eqref{eq1.sf}).
Moreover, for the various balls $B=B(z_Q,r/4)$ under consideration,
the Whitney regions $U_{B,\eps}$ have bounded overlaps,
since each such region is associated to a cube $Q$ with $\ell(Q)\approx r$, and
$B\cap \pom\subset Q$ (in the language of \cite[Sections 3 and 4]{HM-URHM}, $U_{B,\eps}$ is a union of fattened Whitney boxes meeting $B^*$ whose side length is of the order of $\ell(Q)$).  In addition,
$$\sigma(Q) \lesssim \sigma(\Delta)\,.$$
Combining these observations with \eqref{eq2.27}, we obtain the desired packing condition. 
\end{proof}

\section{The geometric proof of Theorem \ref{t1}}

Suppose $\Omega$ is a $(c,C)$-uniform domain with ADR boundary $E:=\pom$. As mentioned
earlier, this is equivalent to being 1-sided NTA, so in particular, we will also use the fact
that $\Omega$ satisfies the corkscrew condition in Definition \ref{def1.cork}.
We can assume, by making numbers smaller if need it, that the constant
$c$  in the definition of the corkscrew condition is the same value as the $c$ in our definition of
$(c,C)$-uniform domains.

\begin{lemma}
There is $\varepsilon>0$ depending only on $c$ and $C$ such that the following holds.
Suppose $x\in \pom$, $r\in(0,{\diam(\pom)})$, and that $P$ is a hyperplane such
that $\bbeta(x,r,P)<\varepsilon$.
Let $v_{P}$ be a unit vector orthogonal to $P$. Define
\[
B^{\pm}(x,r)=B(x,r)\cap \{x+y: \pm y\cdot v_{P}> \varepsilon r\}\subseteq
\ree\setminus \pom.
\]
Then exactly one of either $B^{\pm}(x,r)$ is contained in
$\om_{\rm ext}$ and the other one is contained in $\Omega$.
\label{l:mainlemma}
\end{lemma}

We assume Lemma  \ref{l:mainlemma} for the moment and complete the proof of Theorem \ref{t1}.
Fix $\varepsilon>0$ as in Lemma \ref{l:mainlemma} and $B=B(x,r)$ with
$x\in E$ and $r<{ \diam (\pom)}$. Our aim is to show that $B$ contains
an exterior corkscrew. Set
$(1/2)\Delta=\Delta(x,{r}/{2})$.
By Definition \ref{defi:BWGL} and Theorem \ref{theor:UR-2}, there is $C_{1}$ (depending on $\varepsilon$) so that \eqref{e:carleson} holds.
If we set
\[
\rho=\sup\big\{s<r/2:\exists\,y\in (1/2)\Delta, (y,s)\not\in \widehat{B}_{\varepsilon}\big\},
\]
then
\begin{multline*}
C_{1}\frac{r^{n}}{2^{n}}
\geq \widehat{\sigma}\big(\widehat{B}_{\varepsilon}\cap
(1/2)\Delta\times (0,r/2)\big)
\geq
\widehat{\sigma}\big((1/2)\Delta\times
(\rho,r/2)\big)
\\
=\sigma\big((1/2)\Delta
\big)\log\frac{r}{2\rho}
\geq C^{-1}\frac{r^{n}}{2^{n}}\log\frac{r}{2\rho}\,,
\end{multline*}
where $C^{-1}$ is the constant from \eqref{eq1.ADR}. Thus $\rho\approx r$ (with implicit constants depending on $\varepsilon$ and $n$), hence we may find $x_{1}\in (1/2)\Delta$ and $r\lesssim
r_{1}\leq{r}/{2}$ so that $B(x_{1},r_{1})\subseteq B(x,r)$ and
$\bbeta(x_{1},r_{1})<\varepsilon$.
Lemma \ref{l:mainlemma} and the fact that
$r_{1}\approx r$ imply $B$ has an exterior corkscrew, and this finishes the proof
of Theorem \ref{t1}.

Next we prove Lemma  \ref{l:mainlemma}.

\begin{proof}[Proof of Lemma \ref{l:mainlemma}]
Let $k=1/(2C+1)$, $X^{\pm}=x\pm kr v_{P}/2$, and
$\varepsilon<ck/4$, so that $X^{\pm}\in \ree\setminus \pom$. We will often use the
fact that $k<1$.

\noindent {\it Claim: At least one of $X^{\pm}$ is in $\Omega_{\rm ext}$.} Indeed, assume on the contrary  that
both points are contained in $\Omega$. Then there is a good curve $\gamma\subseteq\Omega$
connecting $X^{\pm}$ such that
\[
\diam (\gamma) \leq \ell(\gamma)\leq C|X^{+}-X^{-}|=Ckr.
\]
By our choice of $k$, and since $X^{+}\in\gamma$, this implies that
\[\gamma\subseteq B(X^{+},Ckr)\subseteq B(x,kr(C+1/2))\subseteq
B(x,r).\]
Furthermore, by  condition (2) in Definition \ref{unif-dom} there must exist $Z\in \gamma\cap P\cap B(x,r)$, such that
\[
\min\{ |Z-X^{+}|,\  |Z- X^-|\} \leq \frac{\delta(Z)}{c}\leq\frac{\varepsilon r}{c}.\]
Assume $|Z-X^+|\le  |Z- X^-|$ and let $x'$ denote the orthogonal projection of $x$ onto $L$. Then
 \[\frac{\varepsilon r}{c}\ge |X^+-Z|\geq |X^{+}-x'|\geq |X^{+}-x|-|x-x'|\geq
 \frac{kr}{2}-\varepsilon r>\frac{kr}{4},\]
which is a contradiction since $\varepsilon<\frac{ck}{4}$.
This proves the claim.

We have proved that either $X^+$ or $X^-$ belong to $\Omega_{\rm ext}$. Suppose for instance that $X^{-}\in \Omega_{\rm ext}$. Note that $B^{-}(x,r)\subseteq
\Omega_{\rm ext}$ since $B^{-}(x,r)$ it is connected, contains $X^{-}$, and does not intersect $\pom$
(because $\bbeta(x,r)<\varepsilon$). Similarly $B^{+}(x,r)\cap \pom=\emptyset$. To show that $B^+(x,r)$ is
contained in $\Omega$, we recall that, by the corkscrew condition, there is
$X_{\Delta}\in B(x,r)\cap\Omega$ such that $B(X_{\Delta},cr)\subseteq B(x,r)\cap \Omega$.
Since $\varepsilon<ck/4<c/4$, we know that $B(X_{\Delta},cr)$
must intersect either $B^{+}(x,r)$ or $B^{-}(x,r)$. Since it is contained in
$\Omega$, it cannot hit $B^{-}(x,r)$, thus it must intersect $B^{+}(x,r)$. Since
$B^{+}(x,r)$ is connected, $B^{+}(x,r)$ must also be contained in $\Omega$ and this completes the proof. Note that, as far as Theorem \ref{t1} is concerned, what is more relevant to us
is that $B^{-}(x,r)\subseteq \Omega_{\rm ext}$, but we have shown $B^{+}(x,r)\subseteq
\Omega$ for the sake of completeness. A similar argument appears in the proof of Theorem 1.18 \cite{Da}.
\end{proof}

\appendix

\section{Proof of Theorem \ref{1-NTA-to-uniform}}

\begin{lemma}
Let $X, X'\in \Omega$, $\rho:=\min\{\delta(X),\delta(X')\}$, and $|X-X'|\leq \Lambda \rho$. Let $B_{1},...,B_{N}$ with $N\leq C(\Lambda)$ be a Harnack chain
with $X\in B_1,\, X'\in B_N,$ $B_k\cap B_{k+1}\neq \emptyset$
and $C_0^{-1}\diam (B_k) \leq \dist (B_k,\partial\Omega)\leq C_0\diam (B_k).$ Then
 $\diam (B_{j})\approx \rho$ with constants depending only on $\Lambda$ and $C_0$.
Let $\gamma$ denote the polygonal
curve connecting $X$ to the center of $B_1$, then the centers of the $B_{j}$'s in order and then the center of $B_N$ to $X'$. Then $\ell(\gamma)\leq M|X-X'|$ where $M$ depends only on $n$, $\Lambda$ and $C_0$.
Moreover for any $Z\in \gamma$ it follows that $\delta(Z)\geq
  c\dist(Z,\{X,X'\})$ for some $c$ that only depends on $n$, $\Lambda$ and $C_0$.
  Thus $\gamma$ is a good curve for $X$ and $X'$.
\label{harnack-lemma}
\end{lemma}

\begin{proof}
Note that if $|X-X'|\le \rho/2$ then the Harnack chain above can be taken to only have one ball and the segment joining $X$ to $X'$ is a good curve.
Thus we assume $|X-X'|\ge \rho/2$.
Since $B_j\cap B_{j+1}\neq \emptyset$ for $j<N$ then
\begin{multline*}
\diam (B_{j})\leq  C_0\dist(B_{j},\partial\Omega)
\leq C_0(\dist(B_{j+1},\partial\Omega)+\diam (B_{j+1}))
\\
\leq C_0(C_0+1)\diam (B_{j+1}),
\end{multline*}
and switching the roles $\diam (B_{j+1})\le C_0(C_0+1)\diam (B_j)$.
Thus for $j=1, \dots, N$,
$\diam (B_{j})\approx \min\{\diam (B_{1}),\diam (B_N)\}\approx \rho$ with comparability constants depending only on $n$, $\Lambda$ and $C_0$.
Note that if $|X-X'|\ge  \rho/2$ then
\[\ell(\gamma)\leq \sum_{i=1}^{N}\diam (B_{i}) + \diam (B_1) +\diam (B_N)\lesssim N
\rho\le C(\Lambda)\,\rho\lesssim |X-X'|
.\]
If $Z\in\gamma$ then there is $i=1,\dots, N$ such that $Z\in B_i$. Assume for instance that $\dist(Z,\{X,X'\})=|Z-X|$, then
\begin{equation}\label{distance}
\dist(Z,\{X,X'\})\leq \diam (B_1)+\sum_{j=1}^i\diam (B_j)\lesssim \rho\lesssim \dist (B_i,\partial\Omega)
\lesssim \delta(Z).
\end{equation}
This eventually shows that $\gamma$ is good curve for $X$ and $X'$  completing the proof.
\end{proof}

\begin{proof}[Proof of Theorem \ref{1-NTA-to-uniform}]

For $X,Y\in\Omega$ we need to find a good curve $\gamma$ connecting $X$ and $Y$.
Let $k\in\Z$ satisfy $2^{k} \leq |X-Y|<2^{k+1}$. Let $j_X, j_Y\in \Z$ be such that
$2^{j_X}\le \delta(X)<2^{j_X+1}$ and $2^{j_Y}\le \delta(Y)<2^{j_Y+1}$. Let $q_X, q_Y\in \partial\Omega$ denote respectively a closest point to $X$ and $Y$ in the boundary, that is, $\delta(X)=|q_X-X|$ and $\delta(Y)=|q_Y-Y|$.

We consider several cases.

{\bf Case 1:} $|X-Y|\leq 1/2\min\{\delta(X),\delta(Y)\}$ then the segment joining $X$ to $Y$ is a good curve.

{\bf Case 2:} $|X-Y|\geq 1/2\min\{\delta(X),\delta(Y)\}$ and $k\leq\min\{j_X+2, j_Y+2\}$ then $|X-Y|<2^{k+1}\le 8\min\{\delta(X),\delta(Y)\}$. By Lemma \ref{harnack-lemma}
there is a good curve $\gamma$ joining $X$ to $Y$.

{\bf Case 3: } $|X-Y|\geq 1/2\min\{\delta(X),\delta(Y)\}$ and $k\geq\min\{j_X+2, j_Y+2\}$. Switching $X$ and $Y$ we may assume that $j_Y+2\geq j_X+2$ and therefore $k\geq j_X+2$. For all $i\in\Z$ with $j_X+2\leq i\leq k$, let $X_i$ be a corkscrew point relative to $B(q_X,2^{i} )\cap\pom$ so that $B(X_{i},c2^{i} )\subseteq B(q_X,2^{i} )$. Provided $c\leq 1/4$, $X_{j_X+2}$ can be chosen to be $X$.

Note that
$|X_{i}-X_{i+1}|\leq 2^{i+2}$
while
$\min\{\delta(X_{i}),\delta(X_{i+1})\}\geq c2^{i} $,
thus there is a Harnack chain $B^i_{1},...,B^i_{N_i}$ with $N_i\leq C(4/c)$ which joins $X_{i}$ and $X_{i+1}$, and so that $X_i\in B^i_1$ and $X_{i+1}\in B^i_{N_i}$. Let $\gamma_{i}$ be the polygonal curve connecting
$X_i$ to $X_{i+1}$ as constructed in Lemma \ref{harnack-lemma}. Note that $\ell(\gamma_i)\lesssim |X_i-X_{i+1}|\lesssim 2^i$. Let $\gamma_X=\cup_{i=j_X+2}^{k-1}\gamma_i$.  Note that $\gamma_X$ joins $X$ to $X_k$ and
\[
\ell(\gamma_X)\leq\sum_{i=j_X+2}^{k-1}\ell(\gamma_i)\lesssim \sum_{i=j_X+2}^{k-1} |X_{i}-X_{i+1}|\lesssim \sum_{i=j_X+2}^{k-1} 2^{i}\lesssim 2^k\leq |X-Y|.
\]
Moreover, note that if $Z\in\gamma_X$, then $Z\in\gamma_i$ for some $i=j_X+2,\dots, k-1$. By Lemma \ref{harnack-lemma} and the construction of $\gamma_i$ we have that $\delta(Z)\geq  c\dist(Z,\{X_i, X_{i+1}\})$.
We may assume $\dist(Z,\{X_i, X_{i+1}\})=|Z-X_i|$ (the other case is treated similarly), then as in \eqref{distance} (here $\rho\approx2^i$) we have that $2^i\lesssim\delta(Z)$, which yields
\begin{multline}\label{dist2}
|Z-X|
\le
|Z-X_i|+|X_i-q_X|+|q_X-X|
\le \frac{1}{c} \delta(Z) + 2^i + \delta(X)
\\ 
\lesssim
\delta(Z)+
2^{j_X +1}
\lesssim \delta(Z).
\end{multline}

To proceed we next observe that
$$
2^{j_Y}
\le
\delta(Y)
\le
|Y-X|+\delta(X)
\le
2^{k+1}+2^{j_X+1}
\le
2^{k+1}+2^{k-1},
$$
which implies that $j_Y\le k+1$. Assume first that $j_Y+2\le k$ (the case $k+1\le j_Y+2\le k+3$ is considered  below). Repeating the previous argument we can find a curve $\gamma_Y$ joining $Y$ and $Y_k$, where $Y_k$ is a corkscrew point relative to $B(q_Y, 2^k)\cap \Omega$ so that $B(Y_{k},c2^{k} )\subseteq B(q_Y,2^{k} )$. This construction gives much as before $\ell(\gamma_Y)\lesssim |X-Y|$ and $|Z-Y|\lesssim\delta(Z)$ for every $Z\in\gamma_Y$.

We also observe that
\begin{multline*}
|X_k-Y_k|
\leq |X_k-q_X| +|q_X-X|+|X-Y|+|Y-q_Y|+|q_Y-Y_k|\\
\leq  2^k + 2^{j_X+1} + 2^{k+1} + 2^{j_Y+1} +2^k\leq 2^{k+3}\le 8|X-Y|
\end{multline*}
and
\[
\min\{\delta(X_k),  \delta(Y_k)\}\geq c 2^k.
\]
Applying Lemma \ref{harnack-lemma} there exists a good curve $\gamma_k$ joining $X_k$ and $Y_k$. Let $\gamma=\gamma_X\cup\gamma_k\cup\gamma_Y$, and note that $\gamma$ joins $X$ and $Y$. Moreover,
since $2^k\le |X-Y|<2^{k+1}$ then
\[
\ell(\gamma)\le \ell(\gamma_X)+ \ell(\gamma_k)+\ell(\gamma_Y)\lesssim |X-Y| + |X_k-Y_k|\lesssim |X-Y|.
\]
For $Z\in \gamma_X\cup \gamma_Y$, \eqref{dist2} and its corresponding version for $Y$ show that $\delta(Z)\ge c\dist(Z,\{X,Y\})$. If $Z\in\gamma_k$, assume for instance that $\dist(Z,\{X_k,Y_k\})=|X_k-Z|$, using the same argument as in \eqref{distance} we have
$2^k\lesssim \delta(Z)$ and hence
\begin{align*}
|Z-X|
&\le |Z-X_k|+ |X_k-q_X|+|q_X-X|
\le
\frac{1}{c} \delta(Z) + 2^k+ 2^{j_X+1}
\lesssim
\delta(Z).
\end{align*}
This proves that $\gamma$ is a good curve for $X$ and $Y$ and completes the case  $j_Y+2\le k$.

Let us finally consider the case $k+1\le j_Y+2\le k+3$. Note that in this situation we clearly have
\begin{align*}
|X_k-Y|
&\leq
|X_k-q_X| +|q_X-X|+|X-Y|
\leq
2^k + 2^{j_X+1} + 2^{k+1} \leq 2^{k+2}\le 4|X-Y|
\end{align*}
and
\[
\min\{\delta(X_k),  \delta(Y)\}\geq c 2^k.
\]
Applying Lemma \ref{harnack-lemma} there exists a good curve $\gamma_k$ joining $X_k$ and $Y$. Let $\gamma=\gamma_X\cup\gamma_k$, and note that $\gamma$ joins $X$ and $Y$. Moreover,
since $2^k\le |X-Y|<2^{k+1}$ we have that
\[
\ell(\gamma)\le \ell(\gamma_X)+ \ell(\gamma_k)\lesssim |X-Y| + |X_k-Y|\lesssim |X-Y|.
\]
For $Z\in \gamma_X$, \eqref{dist2} yields $\delta(Z)\ge c\dist(Z,\{X,Y\})$. If $Z\in\gamma_k$
and $\dist(Z,\{X_k,Y\})=|Z-Y|$ we obtain that $c|Z-Y|\leq\delta(Z)$ from the construction of $\gamma_k$. On the other hand, if
$\dist(Z,\{X_k,Y\})=|X_k-Z|$, using the same argument as in \eqref{distance} we have
$2^k\lesssim \delta(Z)$ and hence
\begin{align*}
|Z-X|
&\le |Z-X_k|+ |X_k-q_X|+|q_X-X|
\le
\frac{1}{c} \delta(Z) + 2^k+ 2^{j_X+1}
\lesssim
\delta(Z).
\end{align*}
This proves that $\gamma$ is a good curve for $X$ and $Y$ and concludes the proof of Theorem \ref{1-NTA-to-uniform}. \end{proof}

\end{document}